\documentclass[12pt,a4paper,oneside]{amsart}
\pdfoutput=1
\usepackage{amssymb}
\usepackage{esint}
\usepackage{mathtools}
\usepackage{hyperref}

\usepackage{mathtools}
\mathtoolsset{showonlyrefs}

\usepackage[english]{babel}
\usepackage[latin1]{inputenc}
\usepackage[pdftex]{graphicx}

\theoremstyle{plain}
\newtheorem{theorem}{Theorem}[section]
\newtheorem{lemma}[theorem]{Lemma}

\newtheorem{proposition}[theorem]{Proposition}

\theoremstyle{definition}

\theoremstyle{remark}
\newtheorem{remark}[theorem]{Remark}

\newcommand{\T}{\mathbb{T}}
\newcommand{\R}{\mathbb{R}}
\newcommand{\Z}{\mathbb{Z}}
\newcommand{\C}{\mathbb{C}}
\newcommand{\N}{\mathbb{N}}

\newcommand{\tf}{\mathcal{T}}
\newcommand{\ip}[2]{\left\langle#1,#2\right\rangle}

\newcommand{\der}{\mathrm{d}}

\renewcommand{\phi}{\varphi}
\newcommand{\abs}[1]{\left| #1 \right|}
\newcommand{\aabs}[1]{\left\| #1 \right\|}

\newcommand{\GL}{\text{GL}}

\DeclareMathOperator{\tr}{tr}

\newcommand{\rt}{R}

\begin{document}

\title{On Radon transforms on compact Lie groups}


\author{Joonas Ilmavirta}
\address{Department of Mathematics and Statistics, University of Jyv\"askyl\"a, P.O.Box 35 (MaD) FI-40014 University of Jyv\"askyl\"a, Finland}
\email{joonas.ilmavirta@jyu.fi}
\urladdr{http://users.jyu.fi/~jojapeil}

\date{\today}

\begin{abstract}
We show that the Radon transform related to closed geodesics is injective on a Lie group if and only if the connected components are not homeomorphic to $S^1$ nor to $S^3$.
This is true for both smooth functions and distributions.
The key ingredients of the proof are finding totally geodesic tori and realizing the Radon transform as a family of symmetric operators indexed by nontrivial homomorphisms from $S^1$.
\end{abstract}



\maketitle

\section{Introduction}

If~$G$ is a compact Lie group, is a function $f:G\to\C$ determined from its integrals over all closed (periodic) geodesics?
This article is concerned with answering this question, and the result is indeed affirmative apart from few special cases, and for those cases we provide counterexamples.
It suffices to consider connected groups, since the recovery in question can be done separately for each connected component; recall that the connected components of a compact Lie group are isometric to each other.
The only compact, connected Lie groups on which the recovery is impossible are the trivial group, $S^1=SO(2)=U(1)$ and $S^3=SU(2)=Sp(1)$.
The only spheres that are Lie groups are precisely~$S^1$ and~$S^3$.

Let~$\tilde\Gamma$ denote the set of periodic geodesics of~$G$.
The Radon transform (or the X-ray transform)~$\rt$ is an integral transform taking a function $f:G\to\C$ to the function $Rf:\tilde\Gamma\to\C$ defined by $Rf(\gamma)=\int_\gamma f$.
Our question is equivalent with asking if~$\rt$ is injective.
This definition obviously makes sense for smooth functions~$f$ and we extend the definition to distributions by duality.
Our injectivity results hold for distributions as well.

To fix the measure of integration in the definition of the Radon transform, we parametrize all periodic geodesics by the interval $[0,1]$ (or equivalently by~$\R/\Z$).
The exact parametrization of geodesics and definition of the Radon transform are given in the beginning of section~\ref{sec:radon}.

The same problem has been considered earlier on the sphere by Funk~\cite{F:S2-radon} and on tori by various authors~\cite{S:radon-variations,AR:radon-torus,I:torus}.
A very similar problem was very recently considered by Grinberg and Jackson~\cite{GJ:radon-symm-space} on symmetric spaces of compact type.
Instead of geodesics, they integrate over maximal tori.
They characterize the kernel of the Radon transform on~$L^2(G)$ and show that the Radon transform is injective if and only if the symmetric space coincides with its adjoint space.
All compact Lie groups are compact symmetric spaces, but not all of them are of compact type.

The Radon transform on the group~$SO(3)$ has applications in texture analysis.
If an object contains lattices in several orientations, one may be interested in the orientation density function $f:SO(3)\to\R$, which is a probability density.
Diffraction measurements give the Radon transform of~$f$, while~$f$ itself is not directly experimentally accessible.
To the best of our knowledge, applications in the study of polycrystalline structures are limited to~$SO(3)$.
The Radon transform on~$SO(3)$ has been studied because of this application (see e.g.~\cite{H:SO3-thesis,BEP:SO3-Radon}).
It turns out that~$SO(3)$ is the only connected, compact Lie group of rank one with injective Radon transform, whence it appears as a special case in the proof of our main result below.

Injectivity of the Radon transform has also been studied on closed manifolds that are not Lie groups or even symmetric spaces.

In particular, there are results on manifolds of negative curvature~\cite{GK:spectral2D,CS:negative-curvature-spectral} and more generally on Anosov manifolds~\cite{DS:anosov} (there are also results for tensor fields~\cite{PSU:anosov,PSU:anosov-mfld}).
Anosov manifolds can be seen as a generalization of manifolds with negative curvature.
Therefore they are very different from compact Lie groups, which never have negative curvature.
The limiting case, the flat torus~$\T^n$, is not an Anosov manifold.

Our methods are close to those used by the author for similar problems on tori~\cite{I:torus}.
The only compact, connected, abelian Lie groups are tori, and it turns out that the general case can be reduced to the abelian one.
We also provide a new proof of the injectivity of~$\rt$ on tori, which is more representation theoretical in nature than the original one.
The key observation made in~\cite{I:torus} is that the Radon transform is symmetric when the geodesics are parametrized in a suitable way, and this remains true when tori are replaced with more general Lie groups.

To uniquely determine what the geodesics are, we need to fix a metric on the Lie group.
%
We take any bi-invariant metric (which is known to exist on all compact Lie groups), whence geodesics are identified with cosets of one parameter subgroups.
It follows that all Lie subgroups are totally geodesic.
Geodesics are thus identified with left cosets of one parameter subgroups.
The Haar measure on~$G$ is the Riemannian volume measure given by the metric and thus bi-invariant.
We remark that on tori any left invariant metric is bi-invariant and all metrics give the same geodesics, so the choice of the metric is does not matter in the abelian case.

We denote the set of smooth functions on~$G$ by $\tf=\tf(G)$ and its dual (the space of distributions) by $\tf'=\tf'(G)$; these spaces are equipped with their usual topologies.
Our main result is the following:

\begin{theorem}
\label{thm:main}
Let~$G$ be a compact, connected Lie group that has at least two elements.
Equip~$G$ with any bi-invariant metric.
The following are equivalent:
\begin{enumerate}
\item\label{item:smooth} The Radon transform is injective on~$\tf(G)$.
\item\label{item:distr} The Radon transform is injective on~$\tf'(G)$.
\item\label{item:S1S3} The group~$G$ is neither~$S^1$ nor~$S^3$.
\end{enumerate}
\end{theorem}

The equivalence \ref{item:smooth}$\iff$\ref{item:S1S3} for semi-simple Lie groups was proven earlier by Grinberg~\cite{G:lie-radon} using methods different from ours.
See remark~\ref{rmk:G} below for details.

The definition of the Radon transform on distributions is given in equation~\eqref{eq:distr-def}.
Note that we have excluded the trivial group by assumption; it is of little interest to consider groups of dimension zero in this context.
We will sketch the proof here and give a detailed proof at the very end of this article.

The implication \ref{item:distr}$\implies$\ref{item:smooth} is trivial and the converse is based on a duality argument.
For the implication \ref{item:smooth}$\implies$\ref{item:S1S3} we only need to give counterexamples on the two exceptional groups.
The hardest part is showing that \ref{item:S1S3}$\implies$\ref{item:smooth}.

The rank of a Lie group is the dimension of its maximal torus.
Therefore all Lie groups of rank two or higher contain a two dimensional torus as a subgroup.
We will show below that if the Radon transform is injective on a subgroup of $G$, then it is injective on $G$ as well.
It then follows that the Radon transform can only fail to be injective when the rank is one.
But there are only three compact, connected Lie groups of rank one:~$S^1$, $S^3$ and~$SO(3)$.
Injectivity of the Radon transform has been shown earlier for~$SO(3)$ because of its applications.

\begin{remark}
Conditions~\ref{item:smooth} and~\ref{item:S1S3} are equivalent under the weaker assumption that the metric is only left (or right) invariant.
Proposition~\ref{prop:subgroup} remains true with one sided invariance and any left invariant metric on a torus is also right invariant.

Our proof of the implication \ref{item:distr}$\implies$\ref{item:smooth} and the definition of the Radon transform on distributions rely on lemma~\ref{lma:symm}, which uses left invariance of the metric and right invariance of the Haar measure.
If one can prove lemma~\ref{lma:symm} for a left invariant metric, the assumption of bi-invariance can be dropped from theorem~\ref{thm:main}.
\end{remark}

Besides injectivity, one naturally wants an explicit reconstruction for the transform~$\rt$.
As the injectivity is based on injectivity on tori, reconstruction can also be done once reconstruction is known on tori.
This reconstruction can be made explicitly on the Fourier side, see~\cite{I:torus}.
The only injectivity result not based on the one on tori is that of~$SO(3)$.
A reconstruction method for this group can be found in~\cite{H:SO3-thesis}.

One of the results in~\cite{GJ:radon-symm-space} was that the kernel of~$\rt$ on~$L^2(S^3)$ is precisely the space of antipodally antisymmetric functions.
If~$\sim$ is the equivalence relation identifying antipodal points on~$S^3$, then $S^3/{\sim}=SO(3)$.
This combined with the arguments of the proof of proposition~\ref{prop:quotient} gives a way to understand why the Radon transform is injective on~$SO(3)$.
The identification $S^3/{\sim}=SO(3)$ shows the projective structure $SO(3)=\R\mathbb P^3$; for a discussion of integral geometry in projective spaces see~\cite{G:projective-int-geom}.

If~$G=S^1$, the kernel of~$\rt$ is exactly the space of functions with zero mean.
This is easy to see since the Radon transform essentially takes the function to its mean.
If~$G$ is the trivial group, there are no geodesics, and the Radon transform of any function is the empty function.

Closed geodesics on~$G$ are, up to translation, nontrivial homomorphisms $\T^1\to G$.
For a nontrivial homomorphism $\gamma:\T^1\to G$ there is a surjective homomorphism $\phi:\T^1\to\T^1$ and an injective homomorphism $\sigma:\T^1\to G$ so that $\gamma=\sigma\circ\phi$.
As the integral over any function on $G$ over $\gamma$ is the same as that over $\sigma$, it is possible to restrict our attention to injective homomorphisms in the construction of geodesics.
This restriction corresponds to assuming that geodesics not only have period one, but that it is their minimal period.
We, however, do not make this restriction.

We can also define the $d$-plane Radon transform on~$G$ if~$G$ has rank~$d$ or higher.
In this case the geodesics are replaced with $d$-dimensional ``planes'', corresponding to homomorphisms $\T^d\to G$ that have image of dimension $d$.
These homomorphisms can also be written as a composition $\T^d\to\T^d\to G$ of a surjective and an injective homomorphism.
If the rank of~$G$ is strictly greater than~$d$, the $d$-plane Radon transform is injective on smooth functions on~$G$ by the arguments given in the proof above.
We will not investigate which groups of rank~$d$ have injective $d$-plane Radon transforms -- this question (on some Lie groups) was studied in~\cite{GJ:radon-symm-space}.

\begin{remark}
\label{rmk:G}
It was recently brought to the author's attention that the Radon transform on compact Lie groups was earlier studied by Grinberg~\cite{G:lie-radon}.
He showed that the Radon transform is injective on the space of continuous functions on compact semi-simple Lie groups other than $SU(2)$ (see~\cite[Corollary]{G:lie-radon}).
The injectivity proof provided by Grinberg does not give a reconstruction method.

In comparison, the main theorem of the present article shows injectivity not only on continuous functions but also distributions.
Moreover, our proof gives an inversion method: from the Radon transform of an unknown function one can recover its Fourier transform on any maximal torus (see~\cite{I:torus} for the explicit recovery of Fourier coefficients from the Radon transform).
Our method of proof is quite different from Grinberg's.
\end{remark}

\section{Some tools and observations}

The following lemma is of great importance to us, since it allows relating the Radon transform on a group to the one on a subgroup.
The result follows from the well known fact that the geodesics with respect to a left (or right) invariant metric are the left (or right) cosets of the nontrivial homomorphic images of~$S^1$.
For results on totally geodesic subgroups, we refer to~\cite{MPMM:totally-geodesic-subgroup}.

\begin{lemma}
\label{lma:totally-geod}
Every Lie subgroup of a compact Lie group is totally geodesic.
\end{lemma}


We will usually denote points on~$G$ by~$x$.
Integration is always done with respect to the bi-invariant Haar measure, which is simply denoted by~$\der x$.

We list below some simple observations about the Radon transform on Lie groups.
All of them are not needed for our result, but we think that they help understanding the nature of the problem.

\begin{proposition}
\label{prop:subgroup}
Suppose~$H$ is a Lie subgroup of a compact connected Lie group~$G$.
If the Radon transform (on smooth functions) is injective on~$H$, it is also injective on~$G$.
\end{proposition}

\begin{proof}
Fix any $g\in G$.
The submanifold~$gH$ is totally geodesic in~$G$ by lemma~\ref{lma:totally-geod}, so the Radon transform of a function $f\in\tf(G)$ on~$G$ gives the Radon transform of $f|_{gH}$ on~$gH$.
Since~$gH$ is isometric to~$H$, the Radon transform is injective on it, so the function~$f$ can be reconstructed on~$gH$ from its Radon transform on~$G$.
Since this holds for all $g\in G$, we have reconstruction on all of~$G$.
\end{proof}

In the case of Lie groups, the following proposition follows directly from the previous one, but we present it in more generality.

\begin{proposition}
\label{prop:prod}
Let~$M$ and~$N$ be Riemannian manifolds.
If the Radon transform is injective on~$M$, then it is also injective on~$M\times N$.
\end{proposition}

\begin{proof}
The proof is rather simple, so we only sketch the main points.

Let~$f$ be a function on~$M\times N$.
Fix any $y\in N$.
For any geodesic~$\gamma$ on~$M$, the function $t\mapsto(\gamma(t),y)$ is a geodesic on~$M\times N$.
Using injectivity of the Radon transform on~$M$ we can reconstruct~$f(\cdot,y)$ from the Radon transform of~$f$ on~$M\times N$.
Repeating this for all $y\in N$ gives~$f$ on all of~$M\times N$.
\end{proof}

The next observation concerns quotients.
For another example of comparison of inverse problems on manifolds and their quotients, see~\cite{C:rigidity-quotient}.

\begin{proposition}
\label{prop:quotient}
Let~$M$ be a closed Riemannian manifold and~$H$ a finite subgroup of the isometry group of~$M$.
If the Radon transform is injective on~$M$, it is also injective on the quotient~$M/H$.
\end{proposition}

\begin{proof}
Let the Radon transforms on~$M$ and~$M/H$ be $\rt_M$ and $\rt_{M/H}$, respectively.
Let $\pi:M\to M/H$ be the quotient map and~$\pi_R^{-1}$ any right inverse of it.
Let $f\in\tf(M/H)$.
We observe that for any geodesic~$\gamma$ on~$M$ we have
\begin{equation}
\rt_M(\pi^*f)(\gamma)
=
\int_\gamma f(\pi(\gamma(t)))\der t
=
\rt_{M/H}f(\pi\circ\gamma)
\eqqcolon
\pi^*(\rt_{M/H}f)(\gamma).
\end{equation}
Thus $\rt_{M/H}f=0$ implies $\rt_M\pi^*f=0$.
By injectivity of~$\rt_M$ we have $\pi^*f=0$, so $f=(\pi^*f)\circ\pi_R^{-1}=0$.
\end{proof}

The result of the proposition above can also be written as $\rt_{M/H}^{-1}=(\pi_R^{-1})^*\rt_M^{-1}\pi^*$.

\section{Fourier analysis on Lie groups}

Fourier analysis on the compact, connected Lie group~$G$ is one of our key tools.
(In fact, we only use it on tori, but we want to give a Fourier analytic view to the full problem at hand.)
If we assumed that~$G$ is abelian, it would follow that~$G$ is a torus, and the classical Fourier analysis on tori is indeed a special case of the more general framework that we will present.
Since the Radon transform on tori was already covered in~\cite{I:torus}, we will never assume that~$G$ is abelian.
Much of the discussion in this section follows the book of Ruzhansky and Turunen~\cite{RT:psdo-symm}.

\subsection{Representation theory}

We recall that a representation of the group~$G$ is a group homomorphism $\rho:G\to\GL(\C^n)$ for some~$n$.
The representation represents the elements of the group as $n\times n$ complex square matrices so that multiplication commutes with~$\rho$.
The number~$n$ is the dimension of the representation~$\rho$.
We restrict our attention to finite dimensional complex representations; as it turns out, these suffice for our purposes.

A representation~$\rho$ is unitary if each~$\rho(g)$, $g\in G$, is a unitary matrix.
A representation~$\rho$ is reducible if we have a splitting $\C^n=V_1\oplus V_2$ so that $\rho(g)V_i=V_i$ for all $g\in G$ for both $i=1,2$ and $0<\dim V_1<n$.
If~$\rho$ is not reducible, it is called irreducible.
Two representations~$\rho_1$ and~$\rho_2$ are equivalent if they have the same dimension and there is an invertible matrix~$A$ such that $\rho_1(g)=A\rho_2(g)A^{-1}$ for all $g\in G$.
To understand all representations of the group~$G$, it often suffices to study the irreducible unitary representations.
For basics of representation theory, see for example~\cite[Part~III]{RT:psdo-symm}.

We let~$\hat G$ denote the set of all irreducible unitary representations of the group~$G$, including only one representation from each equivalence class.
The set~$\hat G$ is not uniquely defined; there is a freedom of making a unitary change of basis for each representations, but this freedom is irrelevant for us.
The object~$\hat G$ is called the dual of~$G$, but it has a natural structure of a group if and only if~$G$ is abelian (in this case~$\hat G$ is called the Pontryagin dual group of~$G$).
The dual of the torus is $\widehat{\T^n}=\{\T^n\ni x\mapsto e^{ik\cdot x}\in U(1);k\in\Z^n\}\approx\Z^n$.
As for the torus, the set~$\hat G$ is always countably infinite.

\subsection{Fourier transform}

Ruzhansky and Turunen~\cite{RT:psdo-symm} define the Fourier transform of a function $f\in L^2(G)$ as a function $\tilde f:Q\to\C$ defined by
\begin{equation}
\tilde f(\rho,i,j)
=
\int_G f(x)\rho_{ij}(x)\der x,
\end{equation}
where
\begin{equation}
Q
=
\{(\rho,i,j);\rho\in\hat G\text{ and }1\leq i,j\leq\dim\rho\}.
\end{equation}
If~$G$ is a torus, this is the usual Fourier series.

We denote by~$\ell^2(Q)$ the space of functions $\phi:Q\to\C$ with the norm $\aabs{\phi}^2=\sum_{(\rho,i,j)\in Q}\dim(\rho)\abs{\phi(\rho,i,j)}^2$.
The most important properties of the Fourier series on tori hold true in this more general setting, namely:

\begin{theorem}[{\cite[Section~7.6]{RT:psdo-symm}}]
\label{thm:fourier}
The Fourier transform is an isometric bijection from~$L^2(G)$ to~$\ell^2(Q)$.
In particular, we have the Parseval formula
\begin{equation}
\int_G f(x)g(x)\der x
=
\sum_{(\rho,i,j)\in Q} \dim(\rho) \tilde f(\rho,i,j) \tilde g(\rho,i,j)
\end{equation}
for all $f,g\in L^2(G)$.
\end{theorem}

We find it more convenient to work with whole matrices instead of their elements, so we will define the Fourier transform~$\tilde f$ of $f\in L^2(G)$ as
\begin{equation}
\tilde f(\rho)
=
\int_G f(x)\rho(x)\der x.
\end{equation}
Since~$\tilde f(\rho)$ is a $\dim\rho\times\dim\rho$-matrix, the Fourier transform is not a mapping to a nice function space.
In our notation the Parseval identity becomes
\begin{equation}
\label{eq:parseval}
\ip{f}{g}
\coloneqq
\int_G f(x)g(x)\der x
=
\sum_{\rho\in\hat G} \dim(\rho) \tilde f(\rho):\tilde g(\rho)
\eqqcolon
\ip{\tilde f}{\tilde g}
\end{equation}
for all $f,g\in L^2(G)$, where $A:B=\tr(AB^T)$ denotes the matrix inner product.
Note that~$\ip{\cdot}{\cdot}$ is not an inner product since it includes no conjugation; this will be more convenient when the same notation is used for the duality between smooth functions and distributions.

We use the same definition for the pairing~$\ip{f}{g}$ when one of the functions~$f$ and~$g$ is matrix valued.
In particular, the Fourier transform is written as
\begin{equation}
\label{eq:fourier-ip}
\tilde f(\rho)
=
\ip{f}{\rho}.
\end{equation}

\subsection{Distributions}

Let $\tf=C^\infty(G)$ denote the space of smooth functions on our group~$G$.
The dual space of~$\tf$ is denoted by~$\tf'$; it is equipped with the weak star topology and referred to as the space of distributions.
We write the duality pairing of $f\in\tf'$ and $g\in\tf$ as~$\ip{f}{g}$ and we use the same notation if $g\in C^\infty(G;\C^{n\times n})$.

The Fourier transform on~$\tf'$ is simply defined by equation~\eqref{eq:fourier-ip}.
Also the Parseval identity~\eqref{eq:parseval} is valid for the duality pairing.
Injectivity of the Fourier transform on~$\tf'$ follows from that on~$\tf$.
Namely, if $f\in\tf'$ satisfies $\tilde f(\rho)=0$ for all $\rho\in\hat G$, then
\begin{equation}
\ip{f}{g}
=
\ip{\tilde f}{\tilde g}
=
0
\end{equation}
for all $g\in\tf$ and consequently $f=0$.

\section{Analysis of the Radon transform}
\label{sec:radon}

Recall that~$\tilde\Gamma$ is the collection of periodic geodesics on~$G$.
We denote by~$\Gamma$ the set of nontrivial homomorphisms $S^1\to G$; thus~$\Gamma$ can be seen as the subset of~$\tilde\Gamma$ that goes through the trivial element $e\in G$.
We scale time so that the period of every geodesic is 1; this amounts to identifying $S^1=\R/\Z$.
Each periodic geodesic on~$G$ is of the form $S^1\ni t\to x\gamma(t)\in G$ for some $x\in G$ and $\gamma\in\Gamma$.
For a geodesic $\gamma\in\Gamma$ the reverse geodesic~$\gamma^{-1}$ is also the pointwise inversion with respect to the group structure.

With this parametrization, we define the Radon transform~$\rt f$ of $f\in\tf$ as a function on $G\times\Gamma$ by
\begin{equation}
\rt f(x,\gamma)
=
\int_0^1f(x\gamma(t))\der t.
\end{equation}
The function $x\mapsto\rt f(x,\gamma)$ is smooth, and in fact the mapping $\tf\ni f\mapsto\rt f(\cdot,\gamma)\in\tf$ is a continuous linear map for every $\gamma\in\Gamma$.


\subsection{Symmetry}

A key property of the Radon transform is that it is symmetric when~$\Gamma$ is seen as a parameter set.
This observation for tori was at the heart of the paper~\cite{I:torus}.
We remark that such symmetry is not meaningful on general closed manifolds.
To be able to translate geodesics, the isometry group of the manifold should at least act transitively.
This is true on Lie groups, but there is also other structure that we make use of.

\begin{lemma}
\label{lma:symm}
For any $f,g\in\tf$ and $\gamma\in\Gamma$ we have
\begin{equation}
\ip{\rt f(\cdot,\gamma)}{g}
=
\ip{f}{\rt g(\cdot,\gamma)}.
\end{equation}
\end{lemma}

\begin{proof}
Using right invariance of the Haar measure yields
\begin{equation}
\begin{split}
\ip{\rt f(\cdot,\gamma)}{g}
&=
\int_0^1\int_Gf(x\gamma(t))g(x)\der x\der t
\\&=
\int_0^1\int_Gf(x)g(x\gamma^{-1}(t))\der x\der t
\\&=
\int_0^1\int_Gf(x)g(x\gamma(t))\der x\der t
\\&=
\ip{f}{\rt g(\cdot,\gamma)}.
\qedhere
\end{split}
\end{equation}
\end{proof}

Lemma~\ref{lma:symm} motivates a weak definition of the Radon transform for distributions.
Namely, for $f\in\tf'$ we let~$\rt f(\cdot,\gamma)$ be the distribution that satisfies
\begin{equation}
\label{eq:distr-def}
\ip{\rt f(\cdot,\gamma)}{g}
=
\ip{f}{\rt g(\cdot,\gamma^{-1})}
\end{equation}
for all $g\in\tf$.

It is also important that for each $\gamma\in\Gamma$ the map $\tf'\ni f\mapsto\rt f(\cdot,\gamma)\in\tf'$ is continuous; this follows from the fact that the corresponding map in~$\tf$ is continuous.

\subsection{Smooth functions and distributions}

We define the convolution of two functions $f,g\in\tf$ via
\begin{equation}
f*g(x)
=
\int_Gf(y)g(yx)\der y.
\end{equation}
This convolution is slightly nonstardard, but it is convenient for us.
For a distribution $f\in\tf'$ we define its convolution with $\eta\in\tf$ by
\begin{equation}
\ip{\eta*f}{g}
=
\ip{f}{\eta*g}
\end{equation}
for all~$g$.

Let $(\eta_k)_{k=1}^\infty$ be a sequence in~$\tf$ such that $\ip{\eta_k}{f}\to f(e)$ as $k\to\infty$ for all $f\in\tf$.
This means that $\eta_k\to\delta_e$ in~$\tf'$, so that $\eta_k*f\to f$ in~$\tf$ for any $f\in\tf$.
It follows easily that also $\eta_k*f\to f$ in~$\tf'$ for any $f\in\tf'$.

It is easy to verify that
\begin{equation}
\label{eq:conv-radon}
(\eta*\rt f(\cdot,\gamma))(x)
=
\rt(\eta*f)(x,\gamma)
\end{equation}
for every $f\in\tf$ and $\gamma\in\Gamma$.
In fact, this holds also for $f\in\tf'$, interpreted in the distributional sense, of course.

\begin{lemma}
\label{lma:smooth-distr}
If~$\rt$ is injective on~$\tf$, it is also so on~$\tf'$.
\end{lemma}

\begin{proof}
Suppose $f\in\tf'$ satisfies $\rt f(\cdot,\gamma)=0$ for all $\gamma\in\Gamma$.
We will show that $f=0$.

We have by~\eqref{eq:conv-radon}
\begin{equation}
0
=
\eta_k*\rt f(\cdot,\gamma)
=
\rt(\eta_k*f)(\cdot,\gamma)
\end{equation}
for all $\gamma\in\Gamma$ and $k\in\N$.
But $\eta_k*f\in\tf$, so by injectivity of~$\rt$ on~$\tf$ we have $\eta_k*f=0$.
By convergence of the convolution, we have $f=\lim_{k\to\infty}\eta_k*f=0$.
\end{proof}

\subsection{Representation integrals}

When we use the following lemma, we have a single geodesic instead of a finite collection.
It would be interesting to find a condition similar to the assumption of the lemma that is equivalent with the statements of theorem~\ref{thm:main}.
We do not know if this is already true for the formulation below.
For a better chance of this being the case, we state it in more generality than needed.
This issue boils down to better understanding of integrals of representations over geodesics.

The following lemma concerns sums of representation integrals, but let us first make a remark about single integrals.
Fix some representation $\rho:G\to GL(\C^n)$ and an injective homomorphism $\gamma:S^1\to G$.
Let us write $C=\gamma(S^1)$ for the circle subrgoup.
We are interested in the integral
\begin{equation}
J
=
\int_{S^1}\rho(\gamma(t))\der t
\in
\C^{n\times n}.
\end{equation}
We normalize~$S^1$ to have unit measure.
If $V\subset\C^n$ is the subspace of invariants under~$\rho|_C$, then~$J$ is a projector from~$\C^n$ onto~$V$.
Thus the matrix~$J$ is invertible if and only if $V=\C^n$, which definitely fails if~$\rho$ is faithful.
On the other hand lemma~\ref{lma:faithless} guarantees invertibility under a condition which is quite the opposite of faithfulness.

\begin{lemma}
\label{lma:rep-int}
Let~$G$ be a compact, connected Lie group.
Suppose that for every unitary, irreducible, finite dimensional representation~$\rho$ of~$G$ there exists a finite collection of periodic geodesics $\gamma_i:[0,1]\to G$ passing through~$e$ and complex numbers~$a_i$, $i=1,\dots,N$, such that the matrix
\begin{equation}
\sum_{i=1}^Na_i\int_0^1\rho(\gamma_i(t))\der t
\end{equation}
is invertible.
Then the Radon transform is injective on~$\tf(G)$.
\end{lemma}

\begin{proof}
Assume that $f\in\tf$ is such that $\rt f(\cdot,\gamma)=0$ for all $\gamma\in\Gamma$.
We need to show that $f=0$.

We have for each $\rho\in\hat G$ and $\gamma\in\Gamma$ by lemma~\ref{lma:symm} and homomorphicity of representations that
\begin{equation}
\label{eq:vv1}
\begin{split}
0
&=
\ip{\rt f(\cdot,\gamma^{-1})}{\rho}
\\&=
\ip{f}{\rt\rho(\cdot,\gamma)}
\\&=
\ip{f}{\rho}
\int_0^1\rho(\gamma(t))\der t.
\end{split}
\end{equation}
Fix any $\rho\in\hat G$.
Let geodesics~$\gamma_i$ and complex numbers~$a_i$, $i=1,\dots,N$, be those given by the assumption.
Summing equation~\eqref{eq:vv1} with weights~$a_i$ and geodesics~$\gamma_i$ we obtain
\begin{equation}
\ip{f}{\rho}\sum_{i=1}^Na_i\int_0^1\rho(\gamma_i(t))\der t
=
0.
\end{equation}
Since the sum gives an invertible matrix, it follows that $\tilde f(\rho)=\ip{f}{\rho}=0$.
By injectivity of the Fourier transform (theorem~\ref{thm:fourier}) this implies that $f=0$.
%
%
\end{proof}

\begin{lemma}
\label{lma:faithless}
Let~$G$ be a compact Lie group and~$\rho$ its representation.
If $\dim\ker\rho\geq1$, then there is a periodic geodesic $\gamma:[0,1]\to G$ with $\gamma(0)=e$ such that $\rho(\gamma(t))=I$ for all $t\in[0,1]$.
In particular, the matrix
\begin{equation}
\int_0^1\rho(\gamma(t))\der t=I
\end{equation}
is invertible.
\end{lemma}

\begin{proof}
Since the representation~$\rho$ is a group homomorphism, its kernel~$\ker\rho$ is a (normal) subgroup of~$G$.
By the dimensional assumption there is a periodic geodesic $\gamma:[0,1]\to\ker\rho$ with $\gamma(0)=e$.
By lemma~\ref{lma:totally-geod} this~$\gamma$ is in fact a geodesic in~$G$.
Since~$\gamma$ lies in the kernel of~$\rho$, the matrix $\rho(\gamma(t))$ is the unit matrix for all~$t$.
\end{proof}

The following theorem was proven in~\cite[Theorem~2]{I:torus}.
We give a new short proof of this theorem.

\begin{theorem}
\label{thm:torus}
If $n\geq2$, the Radon transform is injective on distributions on the torus~$\T^n$.
\end{theorem}

\begin{proof}
By lemmas~\ref{lma:rep-int} and~\ref{lma:faithless} it suffices to show that each $\rho\in\widehat{\T^n}$ has a kernel of dimension one or higher.
The torus is abelian, so all irreducible representations have dimension one.
Since $n\geq2$, this implies that the kernel always has dimension one or higher.
\end{proof}

\subsection{Proof of the main result}

We are now ready to give a detailed proof of theorem~\ref{thm:main}.

\begin{proof}[Proof of theorem~\ref{thm:main}]
The implication \ref{item:distr}$\implies$\ref{item:smooth} is trivial since $\tf\subset\tf'$.
Lemma~\ref{lma:smooth-distr} shows that \ref{item:smooth}$\implies$\ref{item:distr}.
By theorem~\ref{thm:torus} condition~\ref{item:smooth} is true when~$G$ is the torus $\T^n=\R^n/\Z^n$ with $n\geq2$.
(Injectivity of~$\rt$ on tori was also shown in~\cite{I:torus}.)

By proposition~\ref{prop:subgroup} condition~\ref{item:smooth} is true for~$G$ if it is true for any subgroup.
Thus it is true for any~$G$ that contains a torus of dimension two or higher as a subgroup -- or, in other words, its rank is at least two.
The only compact, connected Lie groups not satisfying this condition are the trivial group, $S^1=SO(2)=U(1)$, $S^3=SU(2)=Sp(1)$ and~$SO(3)$; this follows from the classification of compact Lie groups.
But condition~\ref{item:smooth} is known to hold for~$SO(3)$; see~\cite[Theorem~3.3]{H:SO3-thesis} for this result and the rest of the thesis for a thorough discussion of this special case.
Therefore \ref{item:S1S3}$\implies$\ref{item:smooth}.

To conclude the proof, we show \ref{item:smooth}$\implies$\ref{item:S1S3} by providing nontrivial functions in the kernel of~$\rt$ on~$S^1$ and~$S^3$.
On~$S^1$ any function with zero average does the job, and on~$S^3$ we may choose any function that is antisymmetric with respect to the antipodal map (that is, $f\in\tf(S^3)$ with $f(-x)=-f(x)$).
\end{proof}

\section*{Acknowledgements}

The author wishes to thank Eric L. Grinberg, Steven Glenn Jackson, Jere Lehtonen and Mikko Salo for discussions and Jos\'e Figueroa-O'Farrill for a helpful MathOverflow answer.
He is also grateful for feedback from the anonymous referee.
He was partly supported by the Academy of Finland (Centre of Excellence in Inverse Problems Research).

\bibliographystyle{amsplain}
\bibliography{ip}

\end{document}